\documentclass[11pt]{amsart}

\usepackage{amsmath,amssymb,amsthm,multicol,mathtools,hyperref,dsfont,pinlabel,enumitem}
\usepackage[normalem]{ulem}
\usepackage[usenames,dvipsnames]{xcolor}
\usepackage{mathrsfs}

\newcommand\cut{\setminus\!\setminus}
\newcommand\wt{\widetilde}
\newcommand\wh{\widehat}

\newcommand\Z{\mathbb{Z}}

\newcommand\bbm{\begin{bmatrix}}
\newcommand\ebm{\end{bmatrix}}
\newcommand\red[1]{\color{red}#1\color{black}}
\newcommand\white[1]{\color{white}#1\color{black}}
\newcommand\FG[1]{\color{ForestGreen}#1\color{black}}
\newcommand\violet[1]{\color{Violet}#1\color{black}}

\newcommand\LG[1]{\color{gray}#1\color{black}}
\newcommand\DG[1]{\color{darkgray}#1\color{black}}
\newcommand\Cyan[1]{\color{Cyan}#1\color{black}}
\newcommand\bs{\boldsymbol}

\newcommand\MobPos{\raisebox{-2pt}{\includegraphics[height=11pt]{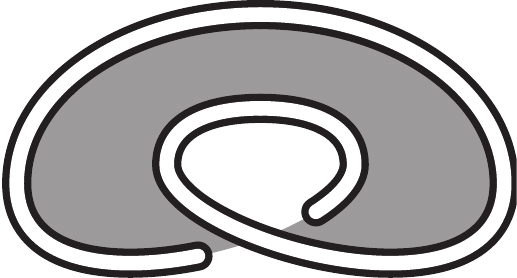}}}
\newcommand\MobNeg{\raisebox{-2pt}{\includegraphics[height=11pt]{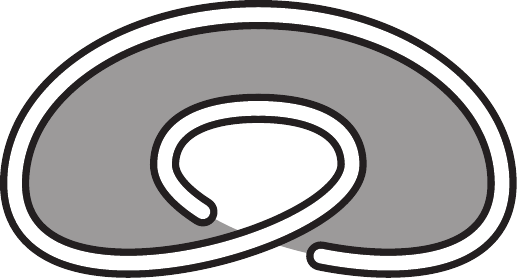}}}

\theoremstyle{plain}
\newtheorem{theorem}{Theorem}[section]

\newtheorem{obs}[theorem]{Observation}
\newtheorem{prop}[theorem]{Proposition}
\newtheorem{cor}[theorem]{Corollary}

\newtheorem*{T:Main1}{Theorem \ref{T:Main1}}
\newtheorem*{T:Main2}{Theorem \ref{T:Main2}}
\newtheorem*{P:Main}{Proposition \ref{P:Main}}

\theoremstyle{definition}

\newtheorem{definition}[theorem]{Definition}
\newtheorem{question}[theorem]{Question}

\theoremstyle{remark}
\newtheorem{rem}{Remark}

\begin{document}

\title[End-essential spanning surfaces]{End-essential spanning surfaces for links in thickened surfaces}

\author{Thomas Kindred}

\address{Department of Mathematics \& Statistics, Wake Forest University \\
Winston-Salem North Carolina, 27109} 

\email{thomas.kindred@wfu.edu}
\urladdr{www.thomaskindred.com}



\maketitle

\begin{abstract}
Let $D$ be a cellular alternating link diagram on a closed orientable surface $\Sigma$. We prove that if $D$ has no removable nugatory crossings then each checkerboard surface from $D$ is $\pi_1$-essential and contains no essential closed curve that is $\partial$-parallel in $\Sigma\times I$. Our chief motivation comes from technical aspects of a companion paper, where we prove that Tait's flyping conjecture holds for alternating virtual links. We also describe possible applications via Turaev surfaces.
\end{abstract}

\section{Introduction}\label{S:intro}

Let $\Sigma$ be a closed orientable surface, not necessarily connected or of positive genus, and let $D\subset \Sigma$ be an alternating diagram of a link $L\subset\Sigma\times I$ such that $D$ cuts $\Sigma$ into disks; such $D$ is said to be {\bf cellular alternating}
.  Then the disks of $\Sigma\cut D$ admit a checkerboard coloring, from which one can construct the {\it checkerboard surfaces} $B$ and $W$ of $D$.\footnote{We denote $I=[-1,1]$. In $\Sigma\times I$, we identify $\Sigma$ with $\Sigma\times\{0\}$ and denote $\Sigma\times\{\pm1\}=\Sigma_\pm$. Also,  $X\cut Y$ denotes $X$ cut along $Y$. Formally, this is the metric closure of $X\setminus Y$. It is homeomorphic to $X\setminus\overset{\circ}{\nu}Y$, but with extra structure from $Y$ encoded in its boundary.} 

These are {\bf spanning surfaces} for $L$: embedded compact surfaces in $\text{int}(\Sigma\times I)$ with no closed components and whose boundary is $L$. We call a spanning surface $F$ for $L$ {\bf end-essential} if it is $\pi_1$-injective and $\partial$-incompressible and contains no essential closed curve which is parallel in $(\Sigma\times I)\cut F$ to $\partial(\Sigma\times I)$. See Definition \ref{D:essential} for details.

A crossing $c$ in $D$ is {\bf removably nugatory} if there is a disk $X\subset \Sigma$ such that $\partial X\cap D=\{c\}$. In that case, one can remove $c$ from $D$ by using a flype to move $c$ across the part of $D$ in $X$ and then using a Reidemeister 1 move undo the resulting monogon. If $D$ has a removable nugatory crossing, then $B$ or $W$ is $\partial$-compressible.  Our main result is the following strong converse of this fact:

\begin{theorem}\label{T:Main2}
If $D\subset\Sigma$ is a cellular alternating link diagram without removable nugatory crossings, then both checkerboard surfaces from $D$ are end-essential.%
\end{theorem}

Our proof strategy for Theorem \ref{T:Main2} is to prove the result with an extra (weak) primeness assumption on $D$ and then extend via connect sum.  We note that in some situations, the conclusion of the theorem follows from work of Ozawa \cite{oz06} and Howie \cite{howie}. 

If $\gamma\subset \Sigma$ is a separating curve such that the annulus $A=\gamma\times I$ intersects the link $L\subset\Sigma\times I$ transversally in two points, then cutting $\Sigma\times I$ along $A$ and gluing on (in the natural way) two 3-dimensional 2-handles, each containing a properly embedded arc, decomposes $(\Sigma,L)$ as an {\bf annular connect sum} $(\Sigma,L)=(\Sigma_1,L_1)\#_\gamma(\Sigma_2,L_2)$. The factors $(\Sigma_i,L_i)$ are uniquely determined by $(\Sigma,L)$ and $\gamma$ up to pairwise homeomorphism (generally, however, the factors $(\Sigma_1,L_i)$ do not determine $(\Sigma,L)$ uniquely). There is also a diagrammatic version of annular connect sum.
See \cite{primes} for details.  

With this setup, if moreover $F$ spans $L$ and $|A\pitchfork F|=1$,\footnote{Here and throughout, $|X|$ denotes the number of components of $X$. The notations $|A\cap F|$ and $|A\pitchfork F|$ carry the same meaning; we use the latter notation if we wish to emphasize or clarify that $A$ and $F$ are transverse.}

\begin{obs}\label{O:1}
Even if $F_i\subset\Sigma_i\times I$ is end-essential for each $i=1,2$, the surface $F_1\natural F_2\subset(\Sigma_1\#\Sigma_2)\times I$ need not be end-essential.  
\end{obs}

\begin{figure}
\begin{center}
\includegraphics[width=\textwidth]{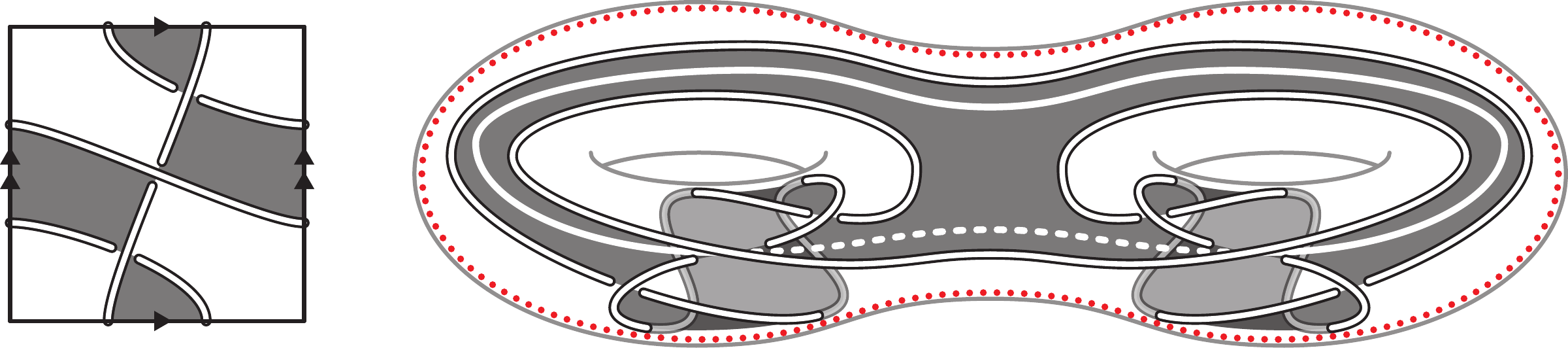}
\caption{The surface $F$ (left) is end-essential, as is its mirror image $F'$, but not $F\natural F'$ (right).}
\label{Fi:PerCom}
\end{center}
\end{figure}

Indeed, consider the example shown in Figure \ref{Fi:PerCom}. The surface $F\subset T^2\times I$ shown left is end-essential (this will follow from Theorem \ref{T:Main1}).  So too is its mirror image $F'\subset T^2\times I$.  Yet, as shown right in the figure, $F\natural F'$ is not end-essential in $(T^2\#T^2)\times I$: the white
curve on $F\natural F'$ is parallel to the dotted curve on $\Sigma_+$. We note that this behavior is related to the following phenomenon in the classical setting:

\begin{obs}\label{O:2}
Even if $F_i\subset S^3$ is $\pi_1$-injective for $i=1,2$, the surface $F_1\natural F_2\subset S^3$ need not be $\pi_1$-injective.
\end{obs}

Indeed, if $M$ and $M'$ are M\"obius bands, each spanning an unknot in $S^2\times I$, then both $M$ and $M'$ are $\pi_1$-injective (but $\partial$-compressible;
in fact, they are the {\it only} connected spanning surfaces in $S^3$ which are $\pi_1$-injective and $\partial$-compressible), and yet $M\natural M'$ is not $\pi_1$-injective. 

The key to dealing with the complication presented by Observation \ref{O:1} is to distinguish between annular connect sums $(\Sigma,L)=(\Sigma,L_1)\#(\Sigma,L_2)$ in general and those which are {\bf local} in the sense that one of $\Sigma_i=S^2$. Indeed, this distinction is at the heart of \cite{primes}.  

Following Howie--Purcell \cite{hp20}, we call a pair $(D,\Sigma)$  {\bf weakly prime} if $D$ is nontrivial and, for any local connect sum decomposition $(\Sigma,D)=(\Sigma,D_1)\#(S^2,D_2)$, either $D_2=\bigcirc$ is the trivial diagram of the unknot, or $(\Sigma,D_1)=(S^2,\bigcirc)$ \cite{hp20}.
Note that no weakly prime, cellular alternating link diagram with more than one crossing has any removable nugatory crossings.


Theorem \ref{T:Main2} will follow from the following two results:

\begin{theorem}\label{T:Main1}
If $D\subset\Sigma$ is a weakly prime, cellular alternating link diagram with more than one crossing, then both checkerboard surfaces from $D$ are  {end-essential}.
\end{theorem}

\begin{prop}\label{P:Main}
Suppose $(\Sigma,L)=(\Sigma_1,L_1)\#(\Sigma_2,L_2)$. If  $F=F_1\natural F_2$ spans $L$, where each $F_i$ is a $\pi_1$-essential spanning surface for $L_i$,  then $F$ is $\pi_1$-essential. If moreover $\Sigma_2=S^2$ and $F_1$ is end-essential, then $F$ is also end-essential.
\end{prop}

Before presenting the proofs, we give a more precise definition of end-essentiality. Note that these properties all concern curves and arcs in $F$ which {\it may self-intersect}. There are alternative notions which do not allow such self-intersections; those notions are sometimes referred to as {\it geometric} (since they indicate whether or not certain types of surgery moves are possible on $F$) and these as {\it algebraic} (due to the equivalence between incompressibility and $\pi_1$-injectivity: see the next remark below) \cite{essence}.

\begin{definition}\label{D:essential}
Let $F$ be a spanning surface for a link $L$ in $\Sigma\times I$, and write $M_F=(\Sigma\times I)\cut F$.
Write $ h_F:M_F\to \Sigma\times I$ for the quotient map that reglues corresponding pairs of points from $\text{int}(F)$ in $\partial M_F$, and denote $\wt{L}={ h_F}^{-1}(L)$, $\wt{\Sigma_\pm}=h_F^{-1}(\Sigma_\pm)$, and $\wt{F}= h_F^{-1}(\text{int}(F))=\partial M_F\setminus (\wt{\Sigma_\pm}\cup \wt{L})$, so that $h_F$ restricts to a homeomorphism $M_F\setminus\wt{F}\to (\Sigma\times I)\setminus\text{int}(F)$ and to a 2:1 covering map $\wt{F}\to\text{int}(F)$.
Then we say that $F$ is:
\begin{enumerate}[label=(\alph*)]
\item {\bf incompressible} if any circle%
\footnote{We use ``circle" as shorthand for ``simple closed curve.'' A circle in a surface is {\it essential} if it does not bound a disk in that surface.}
$\gamma\subset \wt{F}$ that bounds a disk $\wt{X}\subset M_F$  
also bounds a disk in $\wt{F}$. In that case, $X=h_F(X)$ is called a {\it fake compressing disk} for $F$; if $\gamma$ does not bound a disk in $\wt{F}$ then $X$ is a {\it compressing disk} for $F$.
\item {\bf end-incompressible} if 
any circle %
$\gamma\subset \wt{F}$ that is parallel through an annulus $\wt{A}$ in $M_F$ to $\wt{\Sigma_\pm}$ bounds a disk in $\wt{F}$. In that case, $A=h_F(\wt{A})$ is called a {\it fake end-annulus} for $F$; if $\gamma$ does not bound a disk in $\wt{F}$ then $A$ is an {\it end-annulus} for $F$.
\item {\bf $\bs{\partial}$-incompressible} if, for any circle
$\gamma\subset \partial M_F$ with $|\gamma\cap\wt{L}|=1$ that bounds a disk $\wt{X}$ in $M_F$, the arc $\gamma\cut\wt{L}$ is parallel in $\partial M_F\cut\wt{L}$ to $\wt{L}$. If $\gamma\cut\wt{L}$ is not so parallel, then $h_F(\wt{X})$ is a {\it $\partial$-compressing disk} for $F$.
\item {\bf $\pi_1$-essential} if $F$ satisfies (a) and (c).
\item {\bf end-essential} if $F$ satisfies (a), (b) and (c).
\end{enumerate}
\end{definition}

\begin{rem}
In part (a) of Definition \ref{D:essential}, $F$ is incompressible if and only if $F$ is $\pi_1$-injective, meaning that inclusion $\text{int}(F)\hookrightarrow (\Sigma\times I)\setminus L$ induces an injection of fundamental groups (for all possible choices of basepoint).
\end{rem}

\begin{obs}\label{O:3}
Any end-essential surface is $\pi_1$-essential. The converse is true when each component of $\Sigma$ is a 2-sphere. 
\end{obs}

Sometimes, when considering a spanning surface $F$ for a link $L\subset\Sigma\times I$, it is convenient to cut out a regular neighborhood $\nu L$ of the link and view $F$ (which we identify with $F\cut\nu L$) as a properly embedded surface in the link exterior $E=(\Sigma\times I)\cut\nu L$.  This change of perspective affects the aspects of end-essentiality as follows.  Denote $E_F=E\cut F$ and use the natural map $g_F:E_F\to E$ to denote $\wh{\Sigma_\pm}=g_F^{-1}(\Sigma_\pm)$, $\wh{F}=g_F^{-1}(F)$, and $\wh{\partial\nu L}=g_F^{-1}(\partial\nu L)$. Then $F$ is 
\begin{enumerate}[label=(\alph*)]
\item {\bf incompressible} if no essential circle 
$\gamma\subset \wh{F}$ bounds a disk in $E_F$;
\item {\bf end-incompressible} if no essential circle $\gamma\subset\wh{F}$ is parallel in $E_F$ to $\wh{\Sigma_\pm}$; and
\item {\bf $\bs{\partial}$-incompressible} if, for any pair of arcs $\alpha\subset\wh{\partial \nu L}$ and $\beta\subset\wh{F}$ that are parallel in $E_F$, $\alpha$ and $\beta$ are also parallel in $\wh{\partial\nu L}\cup\wh{F}$. %
\end{enumerate}

\section{Proofs of main results}\label{S:PrimeProofs}


\begin{proof}[Proof of Theorem \ref{T:Main1}]
Assume without loss of generality that $\Sigma$ is connected. Denote the checkerboard surfaces of $D$ by $B$ and $W$, where $W$ is the all-$A$ state surface of $D$ and $B$ is the all-$B$ state surface; equivalently, $W$ is the negative-definite surface from $D$, and $B$ is the positive-definite surface \cite{bk20}. View $B$ and $W$ as properly embedded surfaces in the link exterior $(\Sigma\times I)\cut\nu L$.  Denote $B\cap W=v$, which consists of one vertical arc at each crossing.   

Suppose first that $B$ is compressible.  Among all compressing disks for $B$, choose one, $X$, so as 
to minimize $
|X\pitchfork W|)%
$.
Then $X\cap W$ consists only of arcs, each with both endpoints on $v$, since any circle of $X\cap W$ would lie entirely in some disk of $W\setminus v$ and thus bound a disk in $W$; an innermost circle of $X\cap W$ in $W$ would then bound a disk $W'$ of $W\cut X$, and surgering $X$ along $W'$ would give a (sphere and a) compressing disk $X'$ for $F$ with $\partial X'=\partial X$ and $|X'\pitchfork W|<|X\pitchfork W|$, contrary to assumption.

Further, $X\cap W$ is nonempty, since $\partial X$ is essential in $B$,  $v$ cuts $B$ into disks, and each point of $\partial X\cap v$ is an endpoint of an arc of $X\cap W$.  Therefore, there are arcs $\alpha\subset \partial X\cut v$ and $\beta\subset X\cap W$ that cobound an outermost disk $X_0$ of $X\cut W$.  Because $D$ is alternating, $\partial X_0$ appears as shown left in Figure \ref{Fi:Incompressible}, contradicting the fact that $D$ is weakly prime.\footnote{Color guide for Figures \ref{Fi:Incompressible} and \ref{Fi:IncompressibleFinish}: \LG{$\nu L$ light gray}, \DG{$B$ dark gray}.} (One must slide $\partial X_0$ slightly off the crossings to see the contradiction.)

\begin{figure}
\begin{center}
\labellist
\tiny\hair 4pt
\pinlabel {$\violet{\beta}$} [b] at 32 101
\pinlabel {$\white{\alpha}$} [b] at 93 180
\pinlabel {$\violet{\beta}$} [b] at 512 142
\pinlabel {$\red{\alpha}$} [l] at 549 200
\pinlabel {$\violet{\beta}$} [b] at 748 120
\pinlabel {$\red{\alpha'}$} [l] at 743 268
\pinlabel {$\white{\alpha''}$} [l] at 798 135
\endlabellist
\includegraphics[width=\textwidth]{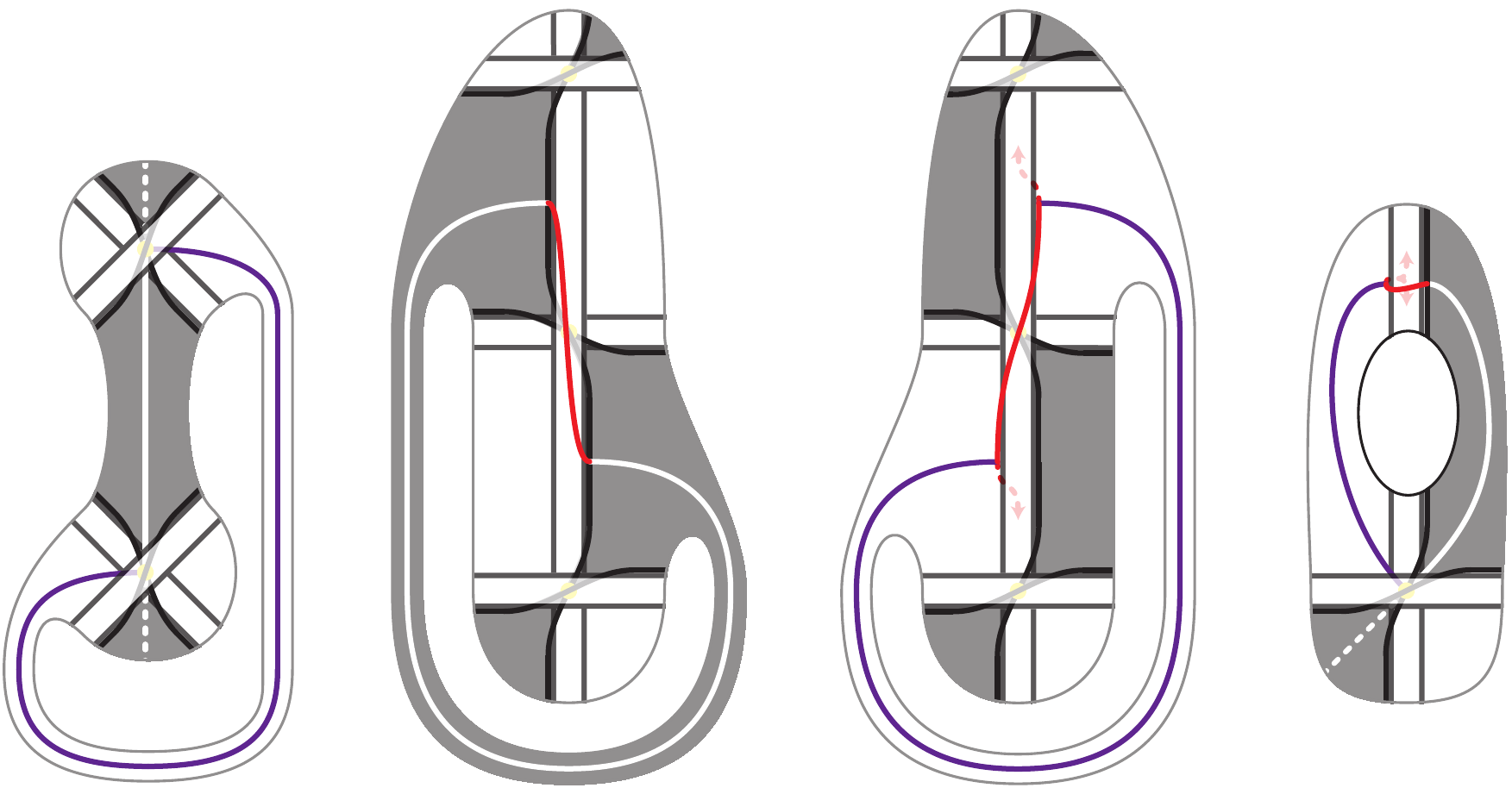}
\caption{Contradictions from the proof of Theorem \ref{T:Main1}.}
\label{Fi:Incompressible}
\end{center}
\end{figure}
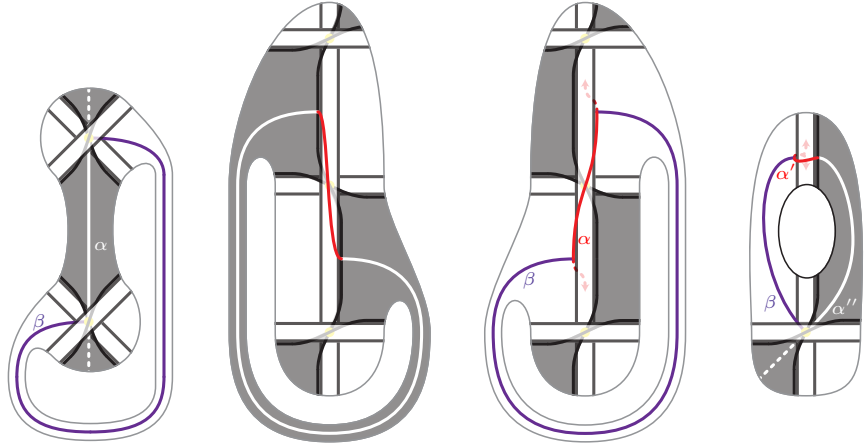

Second, assume that $B$ is end-compressible.  Then without loss of generality there is an end-annulus $N$ whose boundary consists of two closed curves, $\gamma\subset B$ and $\gamma_+\subset\Sigma_+$, where $\gamma$ is essential in $B$ (possibly with self-intersections).
  Among all such end-annuli $N$, choose one which 
  minimizes $
  |N\pitchfork W|
  $. 
Then, as with $X$ and $\partial X$ before, $N\cap W$ consists entirely of arcs with both endpoints on $\gamma\cap v$, and $\gamma$ must intersect $v$. This leads to an outermost disk of $N\cut W$, giving the same contradiction as before, again using Figure \ref{Fi:Incompressible}, left.

Third, assume that $B$ is $\partial$-compressible.  Among all $\partial$-compressing disks for $B$, %
choose one, $X$, so as 
to minimize $
|X\pitchfork W|
$, provided $\partial X\cap\partial v=\varnothing$. Then, for the same reasons as twice before, $X\cap W$ consists only of arcs with endpoints on either $v$ or $\partial \nu L$. Moreover, $X$ must intersect $W$, or else $\partial X$ appears as shown center-left in Figure \ref{Fi:Incompressible}, contradicting the fact that $D$ is weakly prime.

There are at least two outermost disks of $X\cut W$ and just two points of $\partial X\cap\partial B$, neither of them in $X\cap W$ (since $\partial B\cap\partial W=\partial v$ and $\partial X\cap \partial v=\varnothing$), so there is an outermost disk $X'$ of $X\cut W$ that contains at most one of the two points of $\partial X\cap \partial B$. Its boundary consists of an arc $\beta$ of $X\cap W$ and an arc $\alpha\subset\partial X$.  

\begin{figure}
\begin{center}
\includegraphics[width=.6\textwidth]{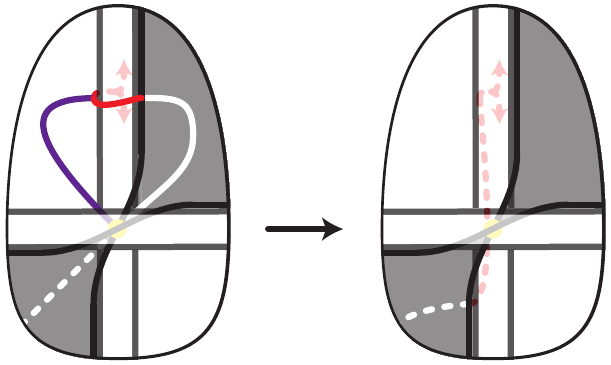}
\caption{This isotopy of $X$ decreases $|X\cap W|$.}
\label{Fi:IncompressibleFinish}
\end{center}
\end{figure}

There are three possibilities for $\alpha$, each giving a contradiction.  If both endpoints of $\alpha$ lie on $v$, then $\alpha\subset B$, giving the same contradiction (left in Figure \ref{Fi:Incompressible}) for a third time. If both endpoints of $\alpha$ lie on $\partial \nu L$, then, since $D$ is alternating, $\alpha\subset\partial\nu L$ appears as shown center-right in Figure \ref{Fi:Incompressible}, %
again contradicting the fact that $D$ is weakly prime. Finally, if one endpoint of $\alpha$ lies on $v$ and the other lies on $\partial \nu L$, then $\alpha$ consists of an arc $\alpha'\subset\partial\nu L$ and an arc $\alpha''\subset B$, and $\partial X'$ must appear as shown right in Figure \ref{Fi:Incompressible}. But then, since $D$ is weakly prime, $\partial X'$ appears as shown left in Figure \ref{Fi:IncompressibleFinish} and therefore, as shown, admits an isotopy which decreases $|X\cap W|$, contrary to assumption.

The same arguments prove that $W$, too, is end-essential.
 \end{proof}

\begin{proof}[Proof of Proposition \ref{P:Main}]
Assume that $(\Sigma,L)=(\Sigma_1,L_1)\#_\gamma(\Sigma_2,L_2)$ and $F=F_1\natural_\gamma F_2$, where both $F_1$ and $F_2$ are $\pi_1$-essential; denote the annulus $A=\gamma\times I$. 

Suppose first that $F$ is compressible. Choose a compressing disk $X$ for $F$ which minimizes $|X\pitchfork  A|$.  The incompressibility of $F_1$ and $F_2$ implies that $X\cap  A\neq\varnothing$, and the minimality of $|X\cap A|$ 
implies that $X\cap A$ contains no circles. Hence, there is an outermost disk $X'$ of $X\cut A$; its boundary consists of an arc $\alpha$ of $X\cap  A$ and an arc $\beta\subset\partial X\cut A$.   

Assume without loss of generality that $\beta\subset F_1$. Viewing each $\Sigma_i\times I$ as a component of $(\Sigma\times I)\cut A$ with a 2-handle attached, $X'$ extends  as shown in Figure \ref{Fi:2HAttach} through this 2-handle in $\Sigma_1\times I$ to give a disk $X''\subset \Sigma_1\times I$ that is either a compressing disk, possibly fake, or a $\partial$-compressing disk for $F_1$.\footnote{Color guide for Figures \ref{Fi:2HAttach} and \ref{Fi:2HAttachEnds} in the online version: \DG{$F$ dark gray}, $L$ black and white, $\LG{X'\text{ and }X''\text{ light gray}}$, $\Cyan{A\text{ light blue}}$, $\red{\alpha \text{ red}}$, $\FG{\Sigma_\pm\text{ green}}$.}  Since $F_1$ is $\pi_1$-essential, the only possibility is that $X''$ is a fake compressing disk for $F_i$. Yet, this implies that $\beta\subset F_1$ is parallel through $F_1$ to $ A$, which contradicts the minimality of $|X\cap  A|$.  

 \begin{figure}
\begin{center}
\includegraphics[width=\textwidth]{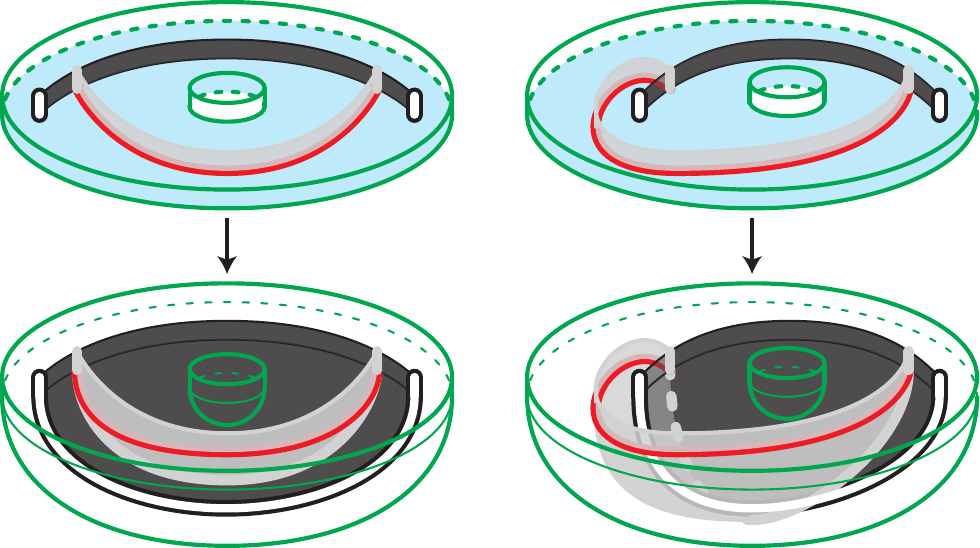}
\caption{A disk $X'$ in a compressing disk for $F_1\natural F_2$ extends through a 2-handle to give a possibly fake compressing disk or $\partial$-compressing disk for $F_1$ or $F_2$.}
\label{Fi:2HAttach}
\end{center}
\end{figure}

If $F$ is $\partial$-compressible, then we may choose a $\partial$-compressing disk $X$ for $F$ which minimizes $|X\pitchfork  A|$, provided  $\partial X\cap \partial\nu L\cap A=\varnothing$.  This last condition, with the same arguments used above, ensures that there is an outermost disk $X'$ exactly as above, the key point being that $\beta$ is disjoint from $\partial\nu L$.  This gives the same contradiction as above.

We have shown that $F$ is $\pi_1$-essential.  Now assume further that $\Sigma_2=S^2$ and $F_1$ is end-essential. Note that $F_2$ is also end-essential, by Observation \ref{O:3}. 

Suppose that $F$ is end-compressible.  Then without loss of generality there is an end-annulus $N$, whose boundary consists of two closed curves, $\omega\subset F$ and $\omega_+\subset\Sigma_+$, where $\omega$ is essential in $F$ (possibly with self-intersections).  Among all such annuli $N$, choose one which minimizes $|N\pitchfork A|$. Then $N\cap A$ is nonempty, since $F_1$ and $F_2$ are end-essential.  Further, $N\cap A$ consists only of arcs, no circles, since any circle in $A\setminus F$ bounds a disk in $(\Sigma_2\times I)\cut F=(S^2\times I)\cut F$, and surgering $N$ along such disks would contradict the minimality of $|N\pitchfork A|$ (using the fact that $F_1$ and $F_2$ are end-essential). 
Moreover, no arc of $N\cap A$ is parallel in $A\cut N$ to $\partial A$ or to $\text{int}(F)$; otherwise, there would be an outermost disk of $A\cut N$ along which to surger $N$, again contradicting minimality. 

\begin{figure}
\begin{center}
\labellist \small \hair 4pt
\pinlabel {$\alpha$} [b] at -11 142
\pinlabel {$\alpha'$} [b] at 90 142
\pinlabel {$\alpha$} [b] at 223 142
\pinlabel {$\alpha'$} [b] at 325 142
\pinlabel {$\alpha$} [b] at 547 142
\pinlabel {$\alpha'$} [b] at 658 142
\pinlabel {$\beta$} [b] at 40 175
\pinlabel {$\beta$} [b] at 275 175
\pinlabel {$\beta'$} [b] at 275 -32
\pinlabel {$\beta$} [b] at 598 175
\pinlabel {$\beta'$} [b] at 598 -32
\pinlabel {$\beta_+$} [b] at 40 74
\pinlabel {$\beta_+$} [b] at 275 74
\pinlabel {$\beta_+'$} [b] at 275 34
\pinlabel {$\beta_+$} [b] at 598 74
\pinlabel {$\beta_+'$} [b] at 598 34
\endlabellist
\includegraphics[width=\textwidth]{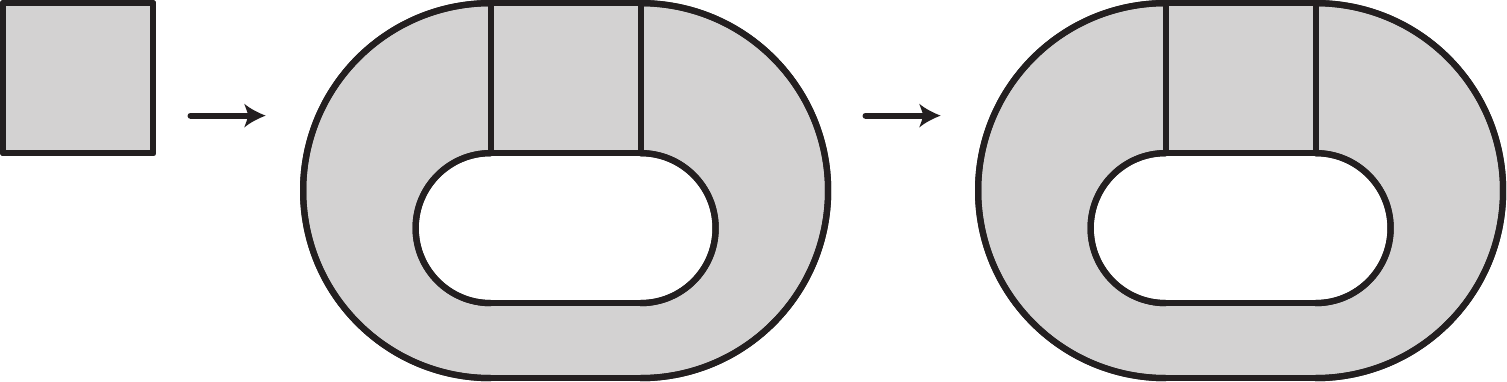}
\caption{Extending $X$ to $Y$ and capping off as $Z$}
\label{Fi:XYZ}
\end{center}
\end{figure}

\begin{figure}
\begin{center}
\includegraphics[width=\textwidth]{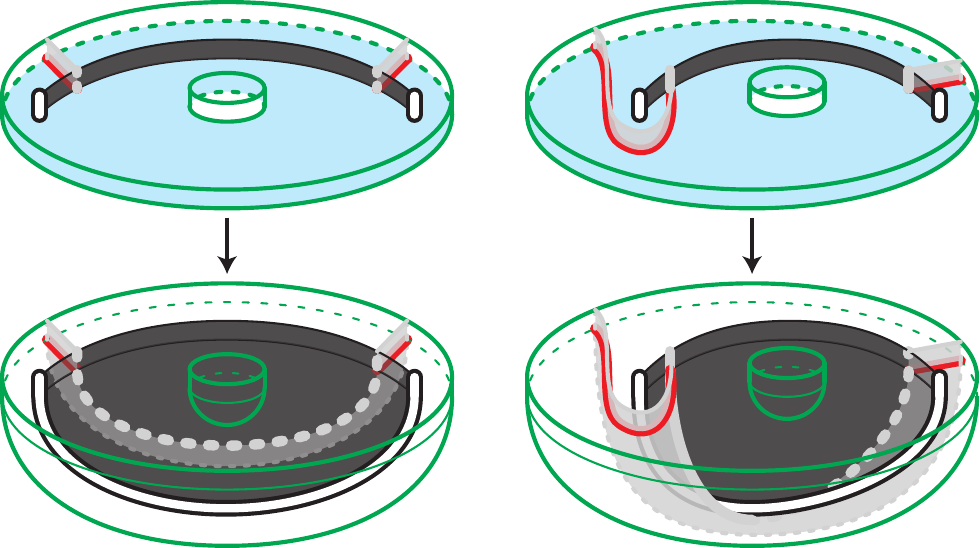}
\caption{Extending a disk $X'$ in an end-annulus for $F_1\natural F_2$ through a 2-handle to get an annulus $Y$}
\label{Fi:2HAttachEnds}
\end{center}
\end{figure}

Hence, all arcs of $N\cap A$ have one endpoint on $F$ and one on $\partial A$, and so each component of $N\cut A$ is a disk whose boundary consists of two arcs of $N\cap A$ together with an arc of $\omega\cut A$ and an arc of $\omega_+\cut A$.  There is at least one disk $X'$ of $N\cut A$ on the same side of $A$ as $F_2$.  Denote the arcs $\partial X'\cap F=\beta$, $\partial X'\cap \Sigma_+=\beta_+\subset\omega_+$, and $\partial X'\cap A=\alpha\sqcup \alpha'$ as far left in Figure \ref{Fi:XYZ}.  There are arcs $\beta'\subset F_2$ and $\beta_+'\subset\Sigma_2^+$ in the attached 2-handle such that $\alpha\cup\beta'\cup \alpha'\cup\beta_+'$ bounds a disk in that 2-handle.  Attaching this disk to $X'$ as in Figure \ref{Fi:XYZ}, left, and Figure \ref{Fi:2HAttachEnds} yields an annulus $Y\subset\Sigma_2\times I=S^2\times I$ with $\partial Y=(\beta\cup\beta')\sqcup(\beta_+\cup\beta'_+)$. Now attach a disk to $Y$ along the component $\beta_+\cup\beta_+'$ of $\partial Y$ that lies in $\Sigma_2^+$, as in Figure \ref{Fi:XYZ}, right, and push the resulting disk $Z$ into the interior of $\Sigma_2\times I$.  Since $F_2$ is $\pi_1$-essential, $Z$ cannot be a compressing disk or a $\partial$-compressing disk; instead, $Z$ must be a {\it fake} compressing disk, implying that $\beta$ is parallel in $F_2$ to $ A$, but this contradicts the minimality of $|N\cap  A|$.
\end{proof}

\begin{proof}[Proof of Theorem \ref{T:Main2}]
Let $D\subset\Sigma$ be a cellular alternating link diagram without removable nugatory crossings. If $D$ is weakly prime, then its checkerboard surfaces are end-essential, by Theorem \ref{T:Main1}. Otherwise, decompose $(\Sigma,D)=(\Sigma,D_0)=(\Sigma,D_1)\#(S^2,D'_1)$ such that $D'_1$ is a prime diagram on $S^2$. Continue in this way, decomposing $(\Sigma,D_i)=(\Sigma,D_{i+1})\#(S^2,D'_i)$, where $D'_i$ is prime on $S^2$ until some $(\Sigma,D_n)$ is weakly prime.  Then $D_n$ is cellular alternating on $\Sigma$, so its checkerboard surfaces are both end-essential by Theorem \ref{T:Main1}.  Likewise, $D'_1,\hdots,D'_{n-1}$ are all cellular alternating on $S^2$, each with at least two crossings, so all of their checkerboard surfaces are $\pi_1$-essential.  Therefore, by Proposition \ref{P:Main}, both checkerboard surfaces from $D$ are end-essential.
\end{proof}

In \cite{essence}, we use a ``twisted'' generalization of Murasugi sum to extend Theorem \ref{T:Main2} to certain state surfaces in thickened surfaces. The state surface construction is exactly as in the classical setting, but with the added restriction that every state circle is contractible on the projection surface. 

\begin{theorem}[Theorems 8.6 and 8.7 of \cite{essence}]
Every homogeneously adequate state surface in a thickened surface is end-essential. In particular, if $D\subset \Sigma$ is a cellular alternating diagram without removable nugatory crossings, then every adequate state surface from $D$ is end-essential in $\Sigma\times I$.
\end{theorem}

\section{Split links and stabilization}

Next, we find that end-essential spanning surfaces behave fairly nicely with respect to split links and (de)stablilization.

A link $L\subset \Sigma\times I$ is {\bf split} if $L$ has a disconnected diagram on $\Sigma$, or equivalently if, for some disjoint union of circles $\gamma\subset \Sigma$, $\Sigma\setminus \gamma$ is disconnected, and $L$ can be isotoped so that it is disjoint from $\gamma\times I$  but intersects each component of $(\Sigma\times I)\setminus(\gamma\times I)$. Note that when each component of $\Sigma$ has positive genus, any link $L\subset\Sigma\times I$ with reducible complement is split, but not every split link has reducible complement.  Note that $\gamma$ may be empty here, as $\Sigma$ need not be connected.

The following result finds motivation in technical aspects of \cite{virtual}:

\begin{prop}\label{P:Split}
If an end-incompressible surface $F$ spans a split link $L\subset \Sigma\times I$, then the boundary of each connected component of $F$ lies in a single split component of $L$.
\end{prop}

\begin{proof}
Let $A=\bigsqcup_tA_t\subset \Sigma\times I$ be a disjoint union of properly embedded annuli, each with one boundary component on each of $\Sigma_\pm$, such that for each component $Y$ of $(\Sigma\times I)\cut A$, $Y\cap L$ is a nonempty nonsplit link in $\Sigma\times I$. Isotope $F$ so that it intersects $A$ transversally and minimally.  This forces $F\cap A\neq \varnothing$: otherwise, $A\cut F$ contains either a disk or an annulus whose boundary is a circle of $F\cap A$ and a component of $\partial A$, and either possibility contradicts the assumptions that $F\pitchfork A$ is minimal and $F$ is end-incompressible. The result follows.
\end{proof}

A pair $(\Sigma,L)$ is {\bf stabilized} if, for some circle 
 $\gamma\subset \Sigma$, $L$ can be isotoped so that it is disjoint from the annulus $\gamma\times I$ but intersects each component of $(\Sigma\times I)\setminus(\gamma\times I)$; one can then {\it destabilize} the pair $(\Sigma,L)$ by cutting $\Sigma\times I$ along $\gamma\times I$ and attaching two 3-dimensional 2-handles in the natural way (this may disconnect $\Sigma$). We say that this destabilization is compatible with a given spanning surface $F\subset \Sigma\times I$ if $F$ can be isotoped to be disjoint from $\gamma\times I$.  In that case, $F$ survives the operation and spans $L$ in $\Sigma'\times I$.

 The reverse operation, $(\Sigma',L)\to (\Sigma,L)$, is called {\it stabilization}: it involves choosing to points $x_0,x_1\in \Sigma'$ with neighborhoods $\nu x_i$ for which $L\cap(\nu x_i\times I)=\varnothing$, cutting out these neighborhoods, and gluing on a thickened annulus in the natural way.  
We say that this stabilization is compatible with a given spanning surface $F\subset \Sigma'\times I$ if $F$ can be isotoped to be disjoint from both $\nu x_i\times I$.  In that case, $F$ survives the operation and spans $L$ in $\Sigma\times I$.

\begin{prop}\label{P:DestabEss}
If  $(\Sigma',L)\to (\Sigma,L)$ is a stabilization which is compatible with an end-essential spanning surface $F'\subset\Sigma'\times I$, then $F'$ is also end-essential in $\Sigma\times I$.
Conversely, if $F\subset \Sigma\times I$ is an end-essential spanning surface for $L$ and $(\Sigma\times I)\setminus L$ is irreducible, then any destabilization $(\Sigma,L)\to (\Sigma',L)$ is compatible with $F$ (but $F$ need not remain end-essential in $\Sigma'\times I$). 
\end{prop}

\begin{proof}
Write $\Sigma\times I=((\Sigma'\times I)\cut(\nu x_0\cup \nu x_1))\cup Y$, where $Y$ is the glued-on thickened annulus. Write $Y'=\partial Y\cut\partial(\Sigma\times I)$: this consists of two parallel, properly embedded annuli in $\Sigma\times I$, and the core curve of each annulus bounds a disk in $(\Sigma\times I)\cut Y$. 

If $F'$ is not end-essential in $\Sigma\times I$, then it has a compressing disk, $\partial$-compressing disk, or end-annulus. Among all the possibilities, choose one, $X$, which minimizes $|X\cap Y'|$. Since $F'$ is end-essential in $\Sigma'\times I= (\Sigma\times I)\cut Y$, $X$ must intersect $Y'$.  By standard innermost/outermost disk arguments, minimality implies that $X\cap Y'$ contains neither inessential circles nor $\partial$-parallel arcs.  Therefore, each component of $Y'$ intersects $X$ in a nonempty collection of  either parallel circles or parallel arcs, and the former is impossible, because then surgering $X$ along any annulus of $Y'\cut X$ would contradict minimality. The latter implies, however, that $\partial X$ intersects each circle of $\partial Y'$, hence intersects both disks that $\partial Y'$ cuts off from $\partial(\Sigma\times I)$, and an outermost arc of $\partial X$ in either disk yields an isotopy which contradicts minimality.

Now consider $F$.  Isotope it to minimize $|F\pitchfork Y'|$.  Then $F$ intersects $Y'$ in circles, no arcs, because $L\cap Y'=\varnothing$.  Minimality and the incompressibility of $F$ imply that no circle of $F\cap Y'$ bounds a disk in $Y'$.  An outermost annulus of $Y'\cut F$ could only be a fake end-annulus for $F$, but this would imply that each component of $\partial Y$ is inessential in $\partial (\Sigma\times I)$, hence that $(\Sigma\times I)\setminus L$ is reducible, contrary to assumption.  Therefore, $F$ is disjoint from $Y'$ and thus compatible with the destabilization.  

\begin{figure}
\begin{center}
\labellist \small \hair 4pt
\pinlabel {$\subset S^2\times I$} at 680 120
\endlabellist
\includegraphics[width=.6\textwidth]{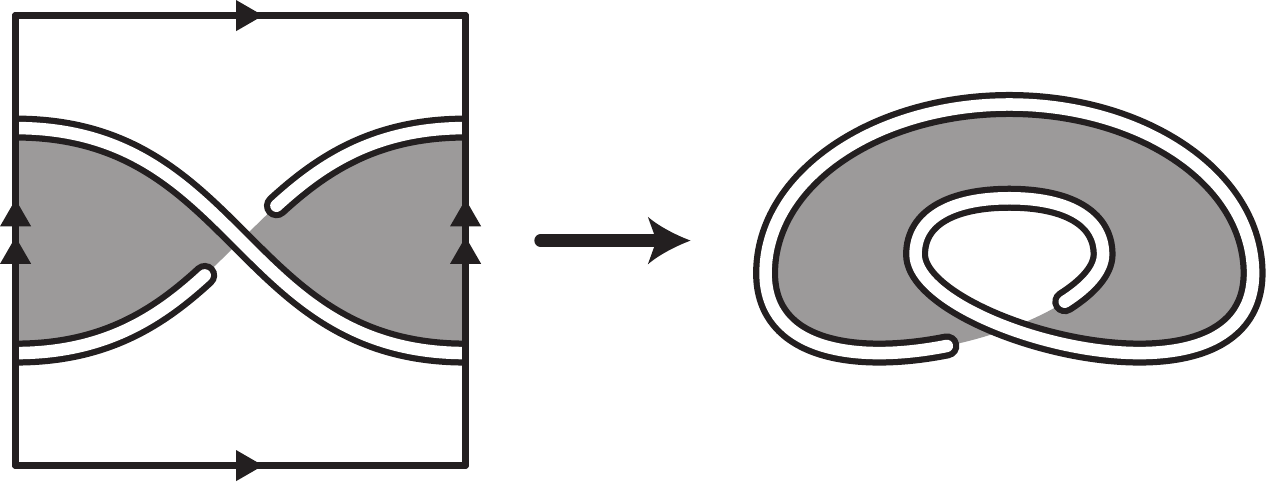}
\caption{Destabilization need not preserve end-essentiality.}
\label{Fi:DeStab}
\end{center}
\end{figure}

Yet, destabilization need not preserve end-essentiality.  For example (see Figure \ref{Fi:DeStab}, let $\gamma$ be a circle on a torus $\Sigma$, and let $F\subset \Sigma\times I$ be a M\"obius band whose core circle is $\gamma$ with a framing of $\frac{1}{2}$.  Then $F$ is end-essential in $\Sigma\times I$, but $F$ is $\partial$-compressible in the thickened 2-sphere obtained by destabilizing.  
\end{proof}

Proposition \ref{P:DestabEss} allows us to extend Theorem \ref{T:Main2} as follows to non-cellular alternating diagrams without nugatory crossings:

\begin{cor}\label{C:EndEss}
 If $M=\Sigma\times I$ is a thickened orientable surface, $D$ is an alternating link diagram on $\Sigma$ without removable nugatory crossings, and $F$ is a checkerboard surface from $D$ whose checkerboard regions are all disks, then $F$ is end-essential in $M$.
 \end{cor}

\begin{proof}
Write $\partial F=L$. If any regions of $\Sigma\cut D$ are not disks, then the associated destabilization operations are compatible with $F$.  Eventually, these operations yield some $M'=\Sigma'\times I$ such that $D$ is cellular alternating on $\Sigma'$ without nugatory crossings.  By Theorem \ref{T:Main2}, $F$ is end-essential in $M'$.  Now reverse the destabilizations. This too is compatible with $F$, so Proposition \ref{P:DestabEss} implies that $F$ is end-essential in $M$.  
\end{proof}

\section{Applications}\label{S:app}


Our chief motivating applications appear in \cite{virtual}, where one of the main results states that any two weakly prime, cellular alternating diagrams of a given link $L\subset\Sigma\times I$ are related by a sequence of flype moves. The idea of the proof is to use Menasco's crossing ball structures to show that the checkerboard surfaces for any two such diagrams are related by a sequence of re-plumbing moves, and that these re-plumbing moves on the checkerboard surfaces correspond to flype moves on the diagrams. The whole approach relies on the results of this paper.  See \cite{virtual} for details.

We conclude this paper with possible applications of a different sort.  The applications all pertain to questions about state surfaces from nontrivial diagrams of trivial links, especially the all-$A$ and all-$B$ state surfaces from such diagrams.  First, we will describe a curious corollary of Theorem \ref{T:Main2}.  Then we will sketch several types of inquiries in which this corollary might prove useful.

Given any connected diagram $D\subset S^2$ of a classical link $L\subset S^3$, construct the all-$A$ state surface $F_A$ and the all-$B$ state surface $F_B$ from $D$ such that near each crossing each state surface has a standard crossing band, and away from these crossing bands $F_A$ and $F_B$ lie entirely on opposite sides of $S^2$.  Then the interiors of $F_A$ and $F_B$ intersect only in vertical arcs, one at each crossing, and a regular neighborhood of $F_A\cup F_B$ is a thickened surface, whose core surface $\Sigma_D$, called the {\bf Turaev surface} of $D$ \cite{tur}, is a Heegaard surface for $S^3$ on which $D$ forms a cellular alternating diagram \cite{dfkls}.  
We note the following consequence of Theorem \ref{T:Main2}:

\begin{cor}\label{C:Turaev}
For any connected diagram $D\subset S^2$ without nugatory crossings, denote the all-$A$ and all-$B$ state surfaces from $D$ by $F_A$ and $F_B$, so that $\nu(F_A\cup F_B)=\Sigma_D\times I$ is a regular neighborhood of the Turaev surface $\Sigma_D$ from $D$.  Then $F_A$ and $F_B$ are end-essential in $\Sigma_D\times I$. Hence, any compressing disk or $\partial$-compressing disk for, say, $F_A$ must intersect $\partial(\Sigma_D\times I)$ (the former at least twice).
\end{cor}

\begin{rem}
Given a  {\it virtual} link diagram $V\subset S^2$, Adams et al \cite{adamsturaev} describe how to construct a Turaev surface $\Sigma_V$, often nonorientable, on which $V$ encodes a cellular alternating diagram of a link $L$ in an orientable $I$-bundle $N$ over $\Sigma_V$. 
One could use this 
construction to extend Corollary \ref{C:Turaev}, but at least two extra hypotheses would be needed: (1) that $\Sigma_V$ is orientable (Theorem 2.5 of \cite{adamsturaev} gives a simple condition on the Gauss code of $V$), so that this paper's results apply, and (2) that it embeds appropriately. For (2), $\Sigma_V$ rarely embeds in $S^3$, but 
the virtual diagram $V$ has an associated diagram $D$ on a closed surface $\Sigma$, and it is more likely that $\Sigma_V$ embeds in $\Sigma\times I$. Still, it does so only if every state circle of the all-$A$ and all-$B$ states of $D$ is contractible on $\Sigma$, and there are 3-crossing cellular knot diagrams on the torus for which this is not true.
\end{rem}

Corollary \ref{C:Turaev} is particularly intriguing when $D$ is a nontrivial diagram of a trivial knot or link, especially one without nugatory crossings.  Then $F_A$ and $F_B$ are both geometrically compressible, and both admit a sequence of geometric compressions and $\partial$-compressions that terminate with the disk.  (This follows from the fact that the only geometrically incompressible spanning surfaces for the unknot are the disk, $\natural_{i=1}^n\MobPos$, and $\natural_{i=1}^n\MobNeg$, $n\in\Z^+$ \cite{jmmz}, none of which arise as such state surfaces from diagrams without nugatory crossings.)  

There are several reasons why it might be interesting to understand such sequences more explicitly. First, any algorithm for obtaining such a sequence would be an unknotting algorithm and thus would detect the unknot.  

Second, Kronheimer--Mrowka used a spectral sequence to prove that Khovanov homology detects the unknot \cite{km}, but there is no elementary proof of that fact, and it remains an open question whether the Jones polynomial detects the unknot.  {\it How} exactly does (almost) everything cancel when one computes the Khovanov homology or Jones polynomial of the unknot from a complicated diagram?  Perhaps a Turaev surface approach as in \cite{dl}, especially with a new geometric perspective, could gain new insight here.

Third, every 2-knot $K$ in $S^4$ admits a {\it bridge trisection}, which describes $K$ via three trivial $2n$-tangles $T_i$, $i\in\Z/3$, in a 3-ball $B_i$, such that each pairwise union $T_i\cup T_{i+1}=L_i$ is a trivial link in the 3-sphere $S_i=B_i\cup B_{i+1}$ \cite{jmmz}.  The triple intersection of the three 3-balls is a 2-sphere $\Sigma$, and cutting $S^4$ along the union of the 3-balls gives three 4-balls $X_i$, where each $\partial X_i=S_i$.  Each $L_i$ bounds a disjoint union $K_i$ of disks, which is unique up to isotopy in $X_i$, and $K=\bigcup K_i$.  A {\it tri-plane} diagram $D=(D_0,D_1,D_2)$ consists of a tangle diagram $D_i$ for each tangle $T_i$.  

If $K$ has normal euler number 0, then one can use $D$ to construct a Seifert solid for $K$ as follows.  Each $D_i\cup D_{i+1}$ is a diagram of an unlink.  Perform Seifert's algorithm on these diagrams (in a compatible way), giving a Seifert surface $F_i$ for each unlink $L_i$.  Take a collar neighborhood $\nu S_i\equiv S_i\times[0,1]$ of $S_i\equiv S_i\times\{0\}$ in $X_i$, and find a sequence of compressions to reduce each $F_i$ to a disjoint union of disks.  Each sequence describes a 
cobordism with boundary in $\nu S_i$ from $F_i\subset S_i\times\{0\}$ to a union of disks in $S_i\times\{1\}$, and the boundary of this cobordism is $F_i\cup K_i$.  The union of the three cobordisms is thus a Seifert solid for $K$, an orientable 3-manifold embedded in $S^4$ with boundary equal to $K$.  See \cite{jmmz} for details. 

If, on each $D_i\cup D_{i+1}$, instead of performing Seifert's algorithm, one instead constructs an arbitrary state surface (again in a compatible way), then nearly the same construction yields a spanning solid for $K$, a 3-manifold (not necessarily orientable) embedded in $S^4$ with boundary equal to $K$. The only difference is that the sequence of moves reducing each $F_i$ to a disjoint union of disks may involve compression {\it and $\partial$-compression.} This construction also works for any surface-link $K\subset S^4$ with normal euler number 0.  The hardest part of both constructions (given a bridge trisection) is finding the sequence of compression (and $\partial$-compression) moves, so again a new geometric perspective here might be useful.

Finally, there are interesting questions one can ask that are completely ``internal'' to the inquiry suggested by Corollary \ref{C:Turaev} (in the case where $D$ is a nontrivial diagram of the unknot without nugatory crossings). Then $F_A$ and $F_B$ are compressible in $S^3\supset\Sigma_D\times I$, but given any compressing disk $X$ for either surface, $X\pitchfork \partial(\Sigma_D\times I)$ is a disconnected multicurve $\omega$ on $\partial(\Sigma_D\times I)$; it is reasonable to require further that (each component of) $\omega$ be essential, and no two components of $\omega$ be parallel, in $\partial(\Sigma_D\times I)$.  What are the possibilities for this multicurve $\omega$? 

In particular, let us denote 
\[\Omega_A(D)=\left\{\omega=X\pitchfork \partial(\Sigma_D\times I)~\left|~\begin{matrix}
X\text{ is a compressing disk for }F_A,\\
\omega\text{ cuts off no disk or annulus}\\
\text{ from }\partial(\Sigma_D\times I)\end{matrix}\right.\right\},\]
and
\[n_A(D)=\min_{\omega\in\Omega_A(D)}|\omega|,\]
and likewise denote $\Omega_B(D)$ and $n_B$. With this setup, we ask:

\begin{question}
How can the sets $\Omega_A(D)$ and $\Omega_B(D)$ relate to each other (depending on $D$ and the link it represents)?
\end{question}

\begin{question}
How do the quantities $|\Omega_A(D)|$, $n_A(D)$ and $\beta_1(F_A)$ relate to each other? 
\end{question}

\begin{question}
How do $|\Omega_A(D)|$, $n_A(D)$, $|\Omega_B(D)|$, and $n_B(D)$ relate to the number of crossings in $D$ and the genus of the Turaev surface? 
\end{question}

Related fundamental questions include:

\begin{question}
If $F$ arises via Seifert's algorithm on a diagram $D$ of a trivial link $L$, must $F$ admit a compression $F\to F'$ such that $F'$ also arises via Seifert's algorithm on a diagram $D'$ of $L$? What if $F$ is an arbitrary state surface from $D$?
\end{question}

\begin{question}
Is every free spanning surface for a trivial link $L$ isotopic to a state surface of some diagram of $L$?
In particular, does every free Seifert surface for a trivial link $L$ arise via Seifert's algorithm?
\end{question}


\begin{thebibliography}{199}

\bibitem[Aetal19]{adamsetal} C. Adams, C. Albors-Riera, B. Haddock, Z. Li, D. Nishida, B. Reinoso, L. Wang, {\it Hyperbolicity of links in thickened surfaces}, Topology Appl. 256 (2019), 262-278.
%
\bibitem[Aetal21]{adamsturaev} C. Adams, O. Eisenberg, J. Greenberg, K. Kapoor, Z. Liang, K. O'Connor, N. Pachecho-Tallaj, Y. Wang, {\it Turaev hyperbolicity of classical and virtual knots}, Algebr. Geom. Topol. 21 (2021), no.7, 3459-3482.
%
%
\bibitem[BCK21]{bck21} H. Boden, M. Chrisman, H. Karimi, {\it The Gordon--Litherland pairing for links in thickened surfaces}, arXiv:2107.00426.


\bibitem[BK22]{bk20} H. Boden, H. Karimi, {\it A characterization of alternating links in thickened surfaces}, Proc. Roy. Soc. Edinburgh Sect. A, 1--19. doi:10.1017/prm.2021.78





%
\bibitem[DFKLS08]{dfkls} O. Dasbach, D. Futer, E. Kalfagianni, X.S. Lin, N. Stoltzfus, {\it The Jones polynomial and graphs on surfaces}, J. Combin. Theory Ser. B 98 (2008), no. 2, 384-399.
%
\bibitem[DL14]{dl} O. Dashbach, A. Lowrance, {\it A Turaev surface approach to Khovanov homology}, 
Quantum Topol. 5 (2014), no. 4, 425-486.%
%
%
%
\bibitem[Ho15]{howie} J. Howie, {\it Surface-alternating knots and links}, PhD thesis, University of Melbourne (2015).
\bibitem[HP20]{hp20} J. Howie, J. Purcell, {\it Geometry of alternating links on surfaces}, Trans. Amer. Math. Soc. 373 (2020), no. 4, 2349-2397.
%
%
%
%
%
\bibitem[JMMZ22]{jmmz} J. Joseph, J. Meier, M. Miller, A. Zupan, {\it Bridge trisections and Seifert solids}, arXiv:2210.09669.
%
\bibitem[Ki22a]{virtual} T. Kindred, {\it The virtual flyping theorem}, arXiv:2210.03720.
%
\bibitem[Ki22b]{primes} T. Kindred, {\it Primeness of alternating virtual links}, arxiv:2210.03225.

%
\bibitem[Ki24]{essence} T. Kindred, {\it How essential is a spanning surface?}, preprint.
%
%
%
%
%
%
%
%
%
%
\bibitem[KM11]{km} P. Kronheimer, T. Mrowka, {\it Khovanov homology is an unknot-detector}, Publ. Math. Inst. Hautes Etudes Sci. (2011), no. 113, 97-208.
%
\bibitem[Oz06]{oz06} M. Ozawa, {\it Nontriviality of generalized alternating knots}, J. Knot Theory Ramifications 15 (2006), no. 3, 351-360.
%
%
%
%
%
%
\bibitem[Tu87]{tur} V.G. Turaev, {\it A simple proof of the Murasugi and Kauffman theorems on alternating links}, Enseign. Math. (2) 33 (1987), no. 3--4, 203--225.%
%
\end{thebibliography}
\end{document}